\documentclass[twoside,11pt]{amsart}

\usepackage{amssymb}
\usepackage{booktabs}
\usepackage{graphicx}
\usepackage{color}

\usepackage{hyperref}
\hypersetup{
pdftitle={A high order positivity preserving DG method for coagulation-fragmentation equations},
pdfauthor={H. Liu, R. Groepler, and G. Warnecke},
pdfsubject={Publication, 2017},
pdfkeywords={population balance equation, aggregation, breakage, discontinuous galerkin method, 
conservation law, particle systems, high order accuracy},
bookmarksopen=true, bookmarksopenlevel=2, breaklinks=true, linktocpage=true,
colorlinks=true, linkcolor=black, citecolor=black, urlcolor=black}

\newtheorem{thm}{Theorem}[section]
\newtheorem{lem}[thm]{Lemma}
\theoremstyle{remark}
\newtheorem{rem}{Remark}[section]

\begin{document}

\title[A DG method for coagulation-fragmentation equations]
{A high order positivity preserving DG method for coagulation-fragmentation equations}

\author[H. Liu \and R. Gr\"opler \and G. Warnecke]
  {Hailiang Liu$^*$ \and Robin Gr\"opler$^{**}$ \and Gerald Warnecke$^{**}$}
\thanks{$^*$Mathematics Department, Iowa State University, Ames, IA 50011
  (\href{mailto:hliu@iastate.edu}{hliu@iastate.edu})}
\thanks{$^{**}$Institute for Analysis and Numerics, Otto-von-Guericke University Magdeburg, 
Universit\"atsplatz 2, 39106 Magdeburg, Germany
(\href{mailto:warnecke@ovgu.de}{warnecke@ovgu.de})}

\date{October 2, 2017}

\begin{abstract} We design, analyze and numerically validate a novel discontinuous Galerkin method
for solving the coagulation-fragmentation equations. The DG discretization is applied to the
conservative form of the model, with flux terms evaluated by Gaussian quadrature with $Q=k+1$
quadrature points for polynomials of degree $k$. The positivity of the numerical solution is
enforced through a simple scaling limiter based on positive cell averages. The positivity of
cell averages is propagated by the time discretization provided a proper time step restriction is imposed. 
\end{abstract}

\maketitle

{\bf Key words.} population balance equation, aggregation, breakage, discontinuous Galerkin method, conservation law, particle systems, high order accuracy \vspace{14pt}

{\bf AMS subject classifications.} 65M60, 65M12, 65R20, 35L65, 82C22 \vspace{14pt}


\section{Introduction} The aggregation-breakage population balance equations (PBEs) are the 
models for the growth of particles by combined effect of aggregation and breakage.
These equations are a type of partial integro-differential equations which are also known as 
coagulation-fragmentation equations. 
These models describe the dynamics of particle growth and the time evolution of a system of 
particles under the combined effect of aggregation, or coagulation, and breakage, or 
fragmentation. Each particle is identified by its size, or volume, which is assumed to be a 
positive real number. From a physical point of view the basic mechanisms taken into account 
are the coalescence of two particles to form a larger one and the breakage of particles into 
smaller ones. These models are of substantial interest in many areas of science and 
engineering.

The equations we consider in this paper describe the time evolution of the particle size 
distribution (PSD) under the simultaneous effect of binary aggregation and multiple breakage. 
The objective of this work is to design a high order discontinuous Galerkin method for these 
equations, so that the numerical solution is highly accurate, and remains non-negative.

In 1917, Smoluchowski \cite{Sm17} proposed the discrete aggregation model in order to describe 
the coagulation of colloids moving according to a Brownian motion which is known as 
Smoluchowski coagulation equation. In 1928, M\"{u}ller \cite{Mu28} provided the continuous 
version of this equation as
\[
\partial_t f(t,x) = \frac{1}{2}\int_0^x K(x-y,y)f(t,x-y)f(t,y)dy
                    - \int_0^\infty K(x,y)f(t,x)f(t,y)dy
\]
with 
\[
f(0,x)=f_0(x). 
\]
Here the variables $x\geq 0$ and $t \geq 0$ denote the size of the particles and time, 
respectively. The number density of particles of size $x$ at time $t$ is denoted by $f(t,x)$. 
The coagulation kernel $K(x, y) \geq 0$ represents the rate at which the particles of size $x$ 
coalesce with particles of size $y$ and is assumed to be symmetric, i.e.\ $K(x, y) = K(y, x)$. 
Later, analogous models for breakage or fragmentation were developed. 
The continuous version of the coagulation and multiple fragmentation equations has been 
investigated, one of the models is of the form 
\begin{align}\label{cmf}
\partial_t f(t,x) =\ & \frac{1}{2}\int_0^x K(x-y,y)f(t,x-y)f(t,y)dy 
                      - \int_0^\infty K(x,y)f(t,x)f(t,y)dy \\
                     & + \int_x^\infty b(x,y)S(y)f(t,y)dy - S(x)f(t,x). \notag
\end{align}
Here the breakage function $b(x,y)$ is the probability density function for the formation of 
particles of size $x$ from the particles of size $y$. It is non-zero only for $x < y$. 
The selection function $S(x)$ describes the rate at which particles of size $x$ are selected 
to break. The selection function $S$ and breakage function $b$ are defined in terms of the 
multiple-fragmentation kernel $\Gamma(x,y)$ as
\[
S(x)=\int_0^x \frac{y}{x}\,\Gamma(x,y)dy, \quad b(x,y)=\Gamma(y,x)/S(y).
\]
This equation is usually referred as the generalized coagulation-fragmentation equation, 
as fragmenting particles can split into more than two pieces. Under some growth conditions 
solutions are shown to exist in the space 
\[
X= \left\{ f\in L^1: \ \int_0^\infty (1+x)f dx <\infty,\ f\geq 0\ \text{a.e.} \right\}
\]
for non-negative initial data $f_0\in X$.

In aggregation-breakage processes the total number of particles varies in time while the 
total mass of particles remains conserved. In terms of $f$, the total number of particles 
and the total mass of particles at time $t\geq 0$ are respectively given by the moments
\[
M_0(t) := \int_0^\infty f(t,x)dx, \quad M_1(t) :=\int_0^\infty xf(t,x)dx.
\]
It is easy to show that the total number of particles $M_0(t)$ decreases by aggregation and 
increases by breakage processes while the total mass $M_1(t)$ does not vary during these 
events. The total mass conservation
\[
M_1(t)=M_1(0)
\]
holds. However, for some special cases of $K$ when it is sufficiently large compared to the 
selection function $S$, a phenomenon called gelation which has to do with a phase transition 
occurs. In this case the total mass of particles is not conserved, see Escobedo et al.\ 
\cite{ELMP03} and further citations for details.
 
Mathematical results on existence and uniqueness of solutions of the equation (\ref{cmf}) 
and further citations can be found in McLaughlin et al.\ \cite{MLM98} and Lamb \cite{La04} 
for rather general aggregation kernels, breakage and selection functions.
However, the equation can only be solved analytically for a limited number of simplified 
problems, see Ziff \cite{Zi91}, Dubovskii et al.\ \cite{DGS92} and the references therein. 
This leads to the necessity of using numerical methods for solving general equations. 
Several such numerical methods have been introduced. Stochastic methods (Monte-Carlo) have been 
developed, see Lee and Matsoukas \cite{LM00} for solving equations of aggregation with binary 
breakage. Finite element techniques can be found in Mahoney and Ramkrishna \cite{WR02} and the 
references therein for the equations of simultaneous aggregation, growth and nucleation.
Some other numerical techniques are available in the literature such as the method of 
successive approximations by Ramkrishna \cite{Ra00}, method of moments \cite{MM04, MF05}, 
finite volume methods \cite{MMG02, KKW13}, and sectional methods, as the Fixed Pivot and 
the Cell Average Technique \cite{Ku06, KR96, Va02}, to solve such PBEs. 
There also exist spline methods \cite{EEE94} and a DG method \cite{Sa06} for the aggregation 
process. All these methods are applied to the standard form of the aggregation-breakage equation 
and have to deal with the problem of mass conservation.

An alternative numerical approach is based on the mass balance formulation. An application of a 
Finite Volume Scheme (FVS) was introduced by Filbet and Lauren\c{c}ot \cite{FL04} for solving 
the aggregation problem. Further, Bourgade and Filbet \cite{BL07} have extended their scheme 
to solve the case of binary aggregation and binary breakage PBEs. For a special case of a 
uniform mesh they have shown error estimates of first order. The scheme has also been extended 
to two-dimensional aggregation problems by Qamar and Warnecke \cite{QW07}. Finally it has been 
observed that the FVS is a good alternative to the methods mentioned above for solving the PBEs 
due to its automatic mass conservation property. An analysis of the finite volume method to 
solve the aggregation with multiple breakage PBEs on general meshes is given in \cite{KKW14}. 
The DG method presented in this work may be seen as a natural extension of the above FVSs.

This paper is organized as follows. First, we derive the DG scheme to solve aggregation-breakage 
PBEs in Section 2. Then in Section \ref{IMP} details of the implementation are given. Later on, the scheme is 
numerically tested for several problems in Section \ref{NUM}. Further, Section \ref{CONC} 
summarizes some conclusions.

\section{Method description}\label{METHOD}
In this section a discontinuous Galerkin (DG) method for solving aggregation-breakage PBEs is 
discussed. Following \cite{FL04} for aggregation, a new form of the breakage PBE is presented 
in order to apply the DG method efficiently.
\subsection{Conservation formulation} Writing the aggregation and breakage terms in divergence 
form enables us to get a precise amount of mass dissipation or conservation. It can be written 
in a conservative form of mass density $n(t, x)=xf(t, x)$, 
\begin{equation}\label{main}
\partial_t n(t,x)+\partial_x (F_a(t,x)+F_b(t,x))=0, \quad n(0,x)=n_0(x)\geq 0\ \text{a.e.},
\end{equation}
where the aggregation flux is 
\[
F_a(t,x)=\int_0^x \int_{x-u}^\infty A(u,v)n(t,u)n(t,v)dv du, \quad A(u,v)=K(u,v)/v,
\]
and the breakage flux is 
\[
F_b(t,x)=-\int_x^\infty \int_{0}^x B(u,v)n(t,v)du dv, \quad B(u,v)=ub(u,v)S(v)/v.
\]
It should be noted that both forms of aggregation-breakage PBEs (\ref{cmf}) and (\ref{main}) 
are interchangeable by using the Leibniz integration rule.
It should also be mentioned that the equation (\ref{main}) reduces to the case of pure 
aggregation or pure breakage process when $F_b$ or $F_a$ is zero, respectively.

In the PBE (\ref{main}) the volume variable $x$ ranges from $0$ to $\infty$. In order to apply 
a numerical scheme for the solution of the equation a first step is to truncate the problem 
and fix a finite computational domain $\Omega:=[0, L]$ for an $0<L< \infty$.

\subsection{DG formulation} Discontinuous Galerkin (DG) methods are a class of finite element 
schemes used to solve mainly conservative PDEs.

We develop a discontinuous Galerkin (DG) method for (\ref{main}) subject to initial data 
$n_0(x)$. Let us partition the interval $\Omega = [0,L]$ into
$0=x_{1/2},x_{3/2},\ldots,x_{N+1/2}=L$ to get $N$ subintervals and denote each cell by 
$I_j = (x_{j-1/2}, x_{j+1/2}]$, $j = 1,\ldots,N$.
Each cell has the length $h_j=x_{j+1/2}-x_{j-1/2}$ and we set $h=\max_j h_j$ to be the mesh size. 
The representative of each cell, usually the center of each cell, 
$x_j = \frac{1}{2}\left( x_{j-1/2} + x_{j+1/2} \right)$, is called pivot or grid point.
The piecewise polynomial space $V_h^k$ is defined as the space of polynomials of degree up 
to $k$ in each cell $I_j$, that is,
\begin{equation}
V_h^k = \{v : v|_{I_j} \in P^k(I_j),\ j = 1,\ldots,N\}.
\end{equation} 
Note that functions in $V_h^k$ are allowed to have discontinuities across cell interfaces.

The DG scheme is defined as follows: find $n_h\in V_h^k$ such that
\begin{equation}\label{dg}
\int_{I_j}\partial_t n_h \phi\,dx - \int_{I_j}F \partial_x \phi\,dx 
+ F\phi \left|_{\partial I_j}\right. = 0, \quad F = F_a+ F_b
\end{equation}
for all test functions $\phi$ in the finite element space $V_h^k$. Here we use the notation, 
$v|_{\partial I_j}=v(x_{j+1/2}^-) -v(x_{j-1/2}^+)$,
and $F_{j+1/2}$ is an appropriate approximation of the continuous flux function $F(t,x_{j+1/2})$. 
In case of a breakage process, the numerical flux may be approximated from the mass flux $F_b$; 
Similarly for the aggregation problem with flux $F_a$. In general, we have 
$F_{j+1/2}=F_{a,j+1/2} + F_{b,j+1/2}$. The initial condition $n_h(0,x) \in V_h^k$ is generated 
by the piecewise $L^2$ projection of $n_0(x)$, that is 
\[
\int_0^L \left(n_h(0,x)-n_0(x)\right)\phi(x)dx=0 \quad \text{for any}\ \phi \in V_h^k. 
\]
The semi-discrete DG scheme (\ref{dg}) is complete.

\subsection{Flux evaluation}\label{subsec:flux}
For the numerical integration of the fluxes we use Gaussian quadrature of order $Q$ with the 
Gauss evaluation points $s_\alpha \in (-1, 1)$ and the weights $\omega_\alpha>0$
\begin{equation}\label{gauss}
\int_{a}^b g(u)du = \frac{b-a}{2}\int_{-1}^1 g\left(\frac{b+a}{2}+\frac{b-a}{2}s\right)ds
= \frac{b-a}{2}\sum_{\alpha=1}^Q \omega_\alpha\,g\left(\frac{b+a}{2}+\frac{b-a}{2}s_\alpha\right)
+R_Q,
\end{equation}
where $R_Q=\mathcal{O}\left((b-a)^{2Q}\right)$ is the approximation residual, which is zero when 
$g$ is a polynomial of degree at most $2Q-1$. We will later use $Q=k+1$, see Remark \ref{remQ}.

First, consider the boundary term in the DG scheme (\ref{dg}). Denote the quadrature points in 
$I_j$ as $\hat x_{j}^\alpha=x_j+\frac{h_j}{2}s_\alpha$ for $\alpha = 1,\ldots,Q$.
Then the aggregation flux 
\[
F_a(t,x_{j+1/2})=\int_0^{x_{j+1/2}}\int_{x_{j+1/2}-u}^{x_{N+1/2}} A(u,v)n_h(t,u)n_h(t,v)dvdu
=\sum_{l=1}^{j} \int_{I_l} n_h(t,u)\Gamma_j(u)du
\]
with the partial flux
\[
\Gamma_j(u)=\int_{x_{j+1/2}-u}^{x_{N+1/2}} A(u,v) n_h(t,v)dv
\]
is approximated by the numerical flux
\begin{equation}\label{faj}
F_{a,j+1/2} = \sum_{l=1}^{j} \frac{h_l}{2}\sum_{\alpha=1}^Q \omega_\alpha 
n_h(t,\hat x_{l}^\alpha)\Gamma_{j,l}^\alpha,
\end{equation}
where $\Gamma_{j,l}^\alpha$ is an approximation to the partial flux $\Gamma_j(\hat x_{l}^\alpha)$ 
applying Gaussian quadrature rules, which we explain below.
Let the index $J$ be chosen such that $x_{j+1/2}-\hat x_{l}^\alpha\in I_{J}$, that is,
$x_{J-1/2} < x_{j+1/2}-\hat x_{l}^\alpha \leq x_{J+1/2}$, then
\[
\Gamma_j(\hat x_{l}^\alpha)=\int^{x_{J+1/2}}_{x_{j+1/2}-\hat x_{l}^\alpha} A(\hat x_{l}^\alpha,v)
n_h(t,v)dv + \sum_{i=J+1}^N \int_{I_i} A(\hat x_{l}^\alpha,v)n_h(t,v)dv.
\]
By the Gaussian quadrature formula (\ref{gauss}), the integral terms can be approximated as
\begin{equation}\label{gja}
\Gamma_{j,l}^\alpha = 
\frac{1}{2}(b_J-a_J) \sum_{\beta=1}^Q\omega_\beta A(\hat x_{l}^\alpha,y_J^\beta) n_h(t,y_J^\beta)
+\sum_{i=J+1}^N \frac{h_i}{2}\sum_{\beta=1}^Q \omega_\beta A(\hat x_{l}^\alpha,\hat x_{i}^\beta)
n_h(t,\hat x_{i}^\beta),
\end{equation}
where $a_J=x_{j+1/2}-\hat x_{l}^\alpha$, $b_J=x_{J+1/2}$ and the quadrature points in the first term are given by
\[
y_J^\beta=\frac{1}{2}(b_J+a_J)+ \frac{1}{2}(b_J-a_J) s_\beta.
\]
To be more precise, the index $J$ depends on $j$, $l$ and $\alpha$. 
Note, that in \eqref{gja} the polynomial function $n_h$ has to be evaluated at intermediate 
points $y_J^\beta$ for the approximation of the integral part here.

Next, we evaluate the breakage flux
\[
F_b(t,x_{j+1/2}) = -\int_{x_{j+1/2}}^{x_{N+1/2}} \int_0^{x_{j+1/2}} B(u,v) n_h(t,v)dudv
= -\sum_{l=j+1}^{N}\int_{I_l} n_h(t,v)G_j(v)dv
\]
with the partial flux
\[
G_j(v)=\sum_{i=1}^j\int_{I_i} B(u,v)du
\]
by the numerical flux
\begin{equation}\label{fjb}
F_{b,j+1/2} = -\sum_{l=j+1}^{N} \frac{h_l}{2} \sum_{\alpha=1}^Q \omega_\alpha 
n_h(t,\hat x_{l}^\alpha) G_{j,l}^\alpha
\end{equation}
where $G_{j,l}^\alpha$ is an approximation to the partial flux $G_j(\hat x_{l}^\alpha)$ and is 
given by
\begin{equation}
G_{j,l}^\alpha = \sum_{i=1}^j \frac{h_i}{2} \sum_{\beta=1}^Q \omega_\beta 
B(\hat x_{i}^\beta,\hat x_{l}^\alpha).
\end{equation}

Now, let us consider the second term in the DG scheme (\ref{dg}). We apply again Gaussian 
quadrature of order $Q$ for the approximation of the integral,
\begin{equation}\label{intF}
\int_{I_j} F\partial_x \phi\,dx  = \int_{-1}^1 F\partial_\xi \phi\,d\xi 
= \sum_{\gamma=1}^Q\omega_\gamma \phi'(s_\gamma) F(t,\hat{x}_j^\gamma).
\end{equation}
The approximation of the flux at the Gauss points, $F(t,\hat{x}_j^\gamma)$, is very similar to 
the approximation of the boundary terms above and is shown in the Appendix.

\begin{rem} \label{remQ}
For the approximation of the integral over a cell $I_j$, we consider the $Q$-point 
Gaussian quadrature which is accurate for polynomials up to degree $2Q-1$. 
In case of a linear and local flux function, $F$ would be of polynomial degree $k$ and 
$\phi'$ has maximal degree $k-1$, in total $2k-1$, corresponding to $Q=k$ Gauss points. 
In our case, the flux is non-local, quadrature rules are also needed for the evaluation 
of the flux. Note that $Q=k$ would be enough to make the scheme consistent. 
However, for our simulations we have chosen $Q=k+1$ since we observe an interesting type 
of superconvergence in this case which we will demonstrate in the numerical tests, 
see Section \ref{NUM}. For simplicity, we choose the same $Q$ for the numerical approximation 
of the partial fluxes which is not necessary.
\end{rem}

\subsection{Time discretization and positivity}\label{sec:time}
By taking the Forward Euler discretization in time to (\ref{dg}), we obtain a fully discrete scheme
\begin{equation}\label{fddg}
\int _{I_j}\frac{n^{m+1}_h-n^m_h}{\Delta t}\phi\,dx - \int _{I_j}F^m\phi_x\,dx 
+ F^m\phi \left|_{\partial I_j}\right.=0.
\end{equation}
Define the cell average of $n_h(x)$ on $I_j$ by
\begin{equation}\label{cellav}
\bar{n}_j := \frac{1}{h_j} \int_{I_j} n_h(x)\,dx.
\end{equation}
Let $\phi=\frac{\Delta t}{h_j}$ so that
\begin{equation}\label{fv}
\bar{n}^{m+1}_j = \bar{n}^m_j - \lambda_j [F^m_{j+1/2}-F^m_{j-1/2}],
\quad \lambda_j=\Delta t/h_j.
\end{equation}
Due to the exactness of quadrature rule for polynomials of degree $2k+1$, and especially of 
degree $k$, we have
\begin{equation}\label{qq}
\bar{n}_j^{m} = \frac{1}{h_j}\int_{I_j} n_h^m(x)dx
= \frac{1}{2} \sum_{\alpha=1}^Q \omega_{\alpha} n_h^m(\hat{x}^\alpha_j).
\end{equation}

For the proof of the following theorem we take a closer look at the differences of the flux at two neighboring interfaces. The flux difference for $F_a$ at time level $m$ can be reorganized to
\[
F_{a,j+1/2}^m-F_{a, j-1/2}^m
= -\sum_{l=1}^{j-1} \frac{h_l}{2} \sum_{\alpha=1}^Q \omega_\alpha n_h^m(\hat x_l^\alpha)
  \left(\Gamma_{j-1,l}^\alpha - \Gamma_{j,l}^\alpha\right)
  +\frac{h_j}{2} \sum_{\alpha=1}^Q \omega_\alpha n_h^m(\hat x_j^\alpha) \Gamma_{j,j}^\alpha,
\]
where the two terms on the right hand side can be seen as a type of birth and death term, respectively. For the use in the following theorem, define the first term as
\[
B_{a,j} = \sum_{l=1}^{j-1} \frac{h_l}{2} \sum_{\alpha=1}^Q \omega_\alpha n_h^m(\hat x_l^\alpha) 
\left(\Gamma_{j-1,l}^\alpha - \Gamma_{j,l}^\alpha\right).
\]
Similarly, the flux difference for $F_b$ at time level $m$ can be combined to
\[
F_{b,j+1/2}^m-F_{b, j-1/2}^m 
= -\sum_{l=j+1}^{N} \frac{h_l}{2} \sum_{\alpha=1}^Q \omega_\alpha n_h^m(\hat x_l^\alpha)
\left(G_{j, l}^\alpha-G_{j-1,l}^\alpha\right)
+ \frac{h_j}{2} \sum_{\alpha=1}^Q \omega_\alpha n_h^m(\hat x_j^\alpha) G_{j-1, j}^\alpha.
\]
Note, that $ G_{j, l}^\alpha-G_{j-1, l}^\alpha\geq 0$ and $G_{j-1, j}^\alpha \geq 0$.

\begin{thm}\label{th2.1}
The high order scheme \eqref{fv} preserves the positivity, i.e., assuming the numerical solution 
at time level $t_m$ positive, then $\bar{n}^{m+1}_j >0$ under the CFL condition
\begin{equation}\label{CFL}
\Delta t < \frac{1}{\max_{j,\alpha} \left( (\Gamma_{j,j}^\alpha)_+ + G_{j-1,j}^\alpha 
+ \frac{1}{h_j\bar{n}^m_j}(-B_{a,j})_+ \right)},
\end{equation}
where $(\cdot)_+$ denotes $\max\{\cdot, 0\}$. In particular, if $n_h^m$ is positive at all points 
$\hat x_j^\alpha$, then $\bar{n}^{m+1}_j>0$.
\end{thm}

\begin{proof} We rewrite scheme \eqref{fv}, using (\ref{qq}), as
\begin{align*}
\bar{n}^{m+1}_j
=\ & \frac{1}{2} \sum_{\alpha=1}^Q \omega_{\alpha} n_h^m(\hat{x}^{\alpha}_j) 
- \lambda_j [F_{j+1/2}-F_{j-1/2}] \\
=\ & \frac{1}{2} \sum_{\alpha=1}^Q \omega_{\alpha} n_h^m(\hat{x}^{\alpha}_j) - \lambda_j \left[ 
  \frac{h_j}{2}\sum_{\alpha=1}^Q \omega_{\alpha} n_h^m(\hat{x}^{\alpha}_j)\Gamma_{j,j}^\alpha 
  - \sum_{l=1}^{j-1} \frac{h_l}{2} \sum_{\alpha=1}^Q \omega_{\alpha} n_h^m(\hat{x}^{\alpha}_l)
  \left(\Gamma_{j-1,l}^\alpha - \Gamma_{j,l}^\alpha\right) \right] \\
& - \lambda_j \left[
  \frac{h_j}{2} \sum_{\alpha=1}^Q \omega_\alpha n_h^m(\hat x_j^\alpha) G_{j-1, j}^\alpha
  - \sum_{l=j+1}^{N} \frac{h_l}{2} \sum_{\alpha=1}^Q \omega_\alpha n_h^m(\hat x_l^\alpha)
  \left(G_{j,l}^\alpha-G_{j-1,l}^\alpha\right) \right] \\
\geq\ & \frac{1}{2} \sum_{\alpha=1}^Q \omega_{\alpha} n_h^m(\hat{x}^{\alpha}_j)
  \left[ 1 - \Delta t\,\Gamma_{j,j}^\alpha - \Delta t\,G_{j-1,j}^\alpha 
  - \Delta t\,\frac{1}{h_j\bar{n}^m_j}(-B_{a,j}) \right].
\end{align*}
Therefore, $\bar{n}^{m+1}_j >0$ under the restriction on the time step (\ref{CFL}).
\end{proof}

\begin{rem}
The CFL condition is somewhat inconvenient, but this condition is only sufficient. 
For aggregation, in most practical cases the term $B_{a,j}$ is positive, hence the third term 
of the condition vanishes. For breakage, the condition depends only on the breakage 
kernel and the maximal grid point, what forces the time step to be very small. Fortunately, in 
most cases the time step can be chosen much larger without losing positivity.
\end{rem}

\begin{rem}
Ideally the assumption should be $n_h \geq0$. In fact, the theorem can be relaxed to 
$n_h^m\geq 0$ at all quadrature points $\hat x_j^\alpha$ requiring only strictly positive 
cell averages, $\bar{n}_j^m>0$.
\end{rem}

\subsection{A scaling limiter}
Theorem \ref{th2.1} implies that in order to preserve the solution positivity, we need to enforce 
$n_h^m(\hat x_j^\alpha)\geq 0$. This is achieved by a reconstruction step using cell averages as 
a reference.

Let $n_h\in P^k(I_j)$ be an approximation of a smooth function $n(x)\geq 0$ with the cell average 
$\bar{n}_j$, defined in (\ref{cellav}). Following the idea of scaling limiter by \cite{ZS10}, we 
define the scaled polynomial by
\begin{equation}\label{c5recon}
\tilde{n}_h(x) = \theta \left(n_h(x)-\bar{n}_j\right) + \bar{n}_j, \quad 
\theta = \min\left\{1, \frac{\bar{n}_j}{\bar{n}_j - \min\limits_{x\in I_j}n_h(x)}\right\}.
\end{equation}
It is easy to check that the cell average of $\tilde n_h$ is still $\bar{n}_j$ and 
$\tilde n_h\geq 0$ in $I_j$. Following \cite{ZS10, LY14}, we have the next lemma.

\begin{lem} \label{c5limaccu}
If $\bar n_j >0$, then the modified polynomial is as accurate as $n_h$ in the following sense
\begin{equation}\label{c5ggh}
|\tilde{n}_h(x) - n_h(x)|\leq C_k\|n_h - n\|_{\infty} \qquad \mbox{for all}\ x \in I_j,
\end{equation}
where $C_k$ is a constant depending on the polynomial degree $k$.
\end{lem}

\begin{rem}
Since we only need control of the values at the Gaussian quadrature points 
\[
S_j=\{\hat x_j^\alpha, \alpha=1, \dots, Q\}
\]
we could replace (\ref{c5recon}) by
\begin{equation}\label{c5zetaS}
\theta = \min\left\{1, \frac{\bar{n}_j}{\bar{n}_j - \min\limits_{x\in S_j}n_h(x)}\right\}
\end{equation}
and the limiter (\ref{c5recon}) with (\ref{c5zetaS}) is sufficient to ensure
\[
\tilde{n}_h(x) \geq 0 \qquad \mbox{for all}\ x\in S_j.
\]
Furthermore, Lemma \ref{c5limaccu} remains valid with this less restrictive limiter, i.e., we have
\[
|\tilde{n}_h(x) - n_h(x)|\leq C_k\|n_h - n\|_{\infty} \qquad \mbox{for all}\ x \in I_j,
\]
where $C_k$ is a constant depending on $k$.
\end{rem}

\begin{rem} 
Lemma \ref{c5limaccu} establishes that the reconstructed polynomial is as accurate as the 
original polynomial. Our numerical results based on this reconstruction are excellent. 
It would be interesting to analyze how the reconstruction error will accumulate in time.
\end{rem} 

\section{Implementation details}\label{IMP}

\subsection{Matrix formulation and implementation}
We would like to introduce the matrix formulation of our numerical scheme and outline the 
flowchart of the algorithm.

As a basis for the set of test functions we choose the Legendre polynomials, 
denoted by $\phi_i(\xi)$,
\[
\phi_0(\xi)=1,\ \ \phi_1(\xi)=\xi,\ \ \phi_2(\xi)=\tfrac{1}{2}\left(3\xi^2-1\right),\dots \quad 
\phi:=(\phi_0, \dots, \phi_k)^T,
\]
which are orthogonal in $L^2([-1,1])$. On each cell, the unknown function can be represented as 
\[
n_h(t,x)=\sum_{i=0}^k n_j^i(t)\phi_i(\xi^j(x)), \quad x\in I_j.
\]
Here $\xi^j(x)$ is the mapping from $I_j$ to $[-1, 1]$, $\xi^j(x):=\frac{2}{h_j}(x-x_j)$. 
To determine $n_h$, it suffices to identify the coefficients $n_j=(n_j^0, \dots, n_j^k)^T$. 
Due to the orthogonality of the basis functions the mass matrix becomes diagonal,
\begin{equation}
\int_{I_j}\partial_t n_h \phi\,dx 
= \frac{h_j}{2}\int_{-1}^1 \phi(\xi)\phi^T(\xi)d\xi\ \frac{d}{dt} n_j(t)
= \frac{h_j}{2}\,\mathrm{diag}\{c_0, \cdots, c_k\}\,\frac{d}{dt} n_j(t),
\end{equation}
where $c_i=\frac{2}{2i+1}$ are the normalization constants.

The fact that we only require $n_h^m$ be nonnegative at certain points can reduce the 
computational cost considerably. Instead of finding the minimum of $n_h$ on the whole 
computational cell $I_j$, we take the minimum only on the test set $S_j$.

\subsection{Algorithm flowchart}
For simplicity, we only consider the Euler forward time stepping. The algorithm with the strong 
stability preserving (SSP) Runge-Kutta method  \cite{GST01} for higher order time discretization can be 
implemented by repeating the following flowchart in each stage; since each SSP-RK method is a 
convex linear combination of the forward Euler. The desired positivity preserving property is 
ensured under a suitable CFL condition.

\begin{enumerate}
\item Initialization: From the given initial data $n_0(x)$

(i) generate $n_h^0 \in V_h^k$ by piecewise $L^2$ projection,

(ii) reconstruct $n_h^0$ as in step 3.

\item Evolution: Use the scheme (\ref{fddg}) to compute $n_h^{m+1}$.

(i) If $\bar n_j^{m+1}$ is positive for all $j$, set $n_h^m = n_h^{m+1}$, continue with step 3.

(ii) Otherwise, halve the time step $\Delta t$ and restart step 2.

\item Reconstruction: Use (\ref{c5recon}) with (\ref{c5zetaS}) and set $n_h^m = \tilde{n}_h^m$, 
  continue with step 2.
\end{enumerate}

\section{Numerical results}\label{NUM}
In this section, we give numerical tests for the proposed positivity-preserving DG scheme 
applied to pure aggregation, breakage and also for the combined processes considering several 
test problems. The following test cases are chosen similar to those in \cite{KKW14}. 
We compare our results with some standard numerical methods, the Cell Average Technique (CAT) 
by J. Kumar \cite{Ku06} and the Finite Volume Scheme (FVS) by Filbet and Lauren\c{c}ot \cite{FL04} 
which is a special case of the presented DG scheme for $k=0$. 
The error calculation for the Cell Average Technique is based on the number density, $f(t,x)$, 
whereas all other errors are calculated with respect to the mass density, $n(t,x)=xf(t,x)$.

For this type of equations it is convenient to use a geometric grid, $x_{j+1/2}=rx_{j-1/2}$. 
We choose the factor $r=2^{(30/N)}$ to span about 9 orders of magnitude. We set $x_{1/2}=0$ 
and $x_{3/2}=x_0$ where $x_0$ can be adapted to the different cases.

All numerical simulations below were carried out to investigate the experimental order of 
convergence (EOC). The error is measured in a continuous and a discrete norm. 
The continuous $L^1$ norm can be approximated by
\begin{align}
e_h = \sum_j \frac{h_j}{2} \sum_{\alpha=1}^R \tilde\omega_\alpha 
|n_h(t,\tilde x_j^\alpha)-n(t,\tilde x_j^\alpha)|,
\end{align}
using a high-order Gaussian quadrature rule, where $\tilde\omega_\alpha>0$ are the weights, and 
$\tilde x_j^\alpha$ are the corresponding Gauss points in each cell $I_j$. 
We choose $R=16$ through all the examples. 
The symbol $h$ corresponds to the number of cells.
The discrete $L^1$ norm is given by
\begin{align}
e_{h,d} = \sum_j \frac{h_j}{2} \sum_{\alpha=1}^Q \omega_\alpha 
|n_h(t,\hat x_j^\alpha)-n(t,\hat x_j^\alpha)|,
\end{align}
where the Gaussian quadrature points are the same as used for the discretization of the scheme.
We have chosen the $L^1$ norm since it is a natural choice for conservation laws. Using the $L^2$ 
or $L^\infty$ norm, similar numerical results and the same order of convergence can be obtained. 
The error calculation in the discrete norm coincides for $k=0$ with that used by Filbet and 
Lauren\c{c}ot \cite{FL04} for the Finite Volume Scheme.

If the problem has analytical solutions, the following formula is used to calculate the EOC
\[
\mathrm{EOC} = \ln(e_h /e_{h/2} )/ \ln(2).
\]
In case of unavailability of the analytical solutions, the EOC can be computed using 
$e_h=\|n_{h}-n_{h/2}\|$ where $n_{h/2}$ is interpolated onto the grid of $n_h$ by a high order 
interpolation method.
For the calculation of the EOC we show the numerical errors at time $t=0.01$ since the order of 
convergence of the DG scheme after longer times is disturbed by the low order time integration.

We shall also evaluate the moments of the numerical solution by 
\[
M_{p,h} = \int_0^L x^{p-1} n_h(t,x)dx
= \sum_j \frac{h_j}{2} \sum_{\alpha=1}^R \tilde\omega_\alpha (\tilde x_j^\alpha)^{p-1} 
n_h(t,\tilde x_j^\alpha), \quad p=0, 1, 2, \dots,
\]
where $R$ is chosen large enough such that the integration is exact for all moments under 
consideration. The error in the moments is then given by
\[
e(M_{p,h}) = \frac{|M_{p,h}-M_p|}{M_p}, \quad p=0, 1, 2, \dots
\]
where the error in the first moment $e(M_{1,h})$ of course vanishes by construction of the method.

\subsection{Pure aggregation} {\bf Test case 1:} \\
The numerical verification of the EOC of the DG solutions for aggregation is discussed by taking three 
problems, namely the case of constant, sum and product aggregation kernels. The analytical 
solutions for these problems taking the exponential initial distribution $n(0, x)=x\exp(-x)$ 
have been given in Scott \cite{Sc68}.
The computational domain in these cases is taken as $[10^{-3}, 10^6]$.

\begin{figure}[!ht]
\includegraphics[width=7.5cm]{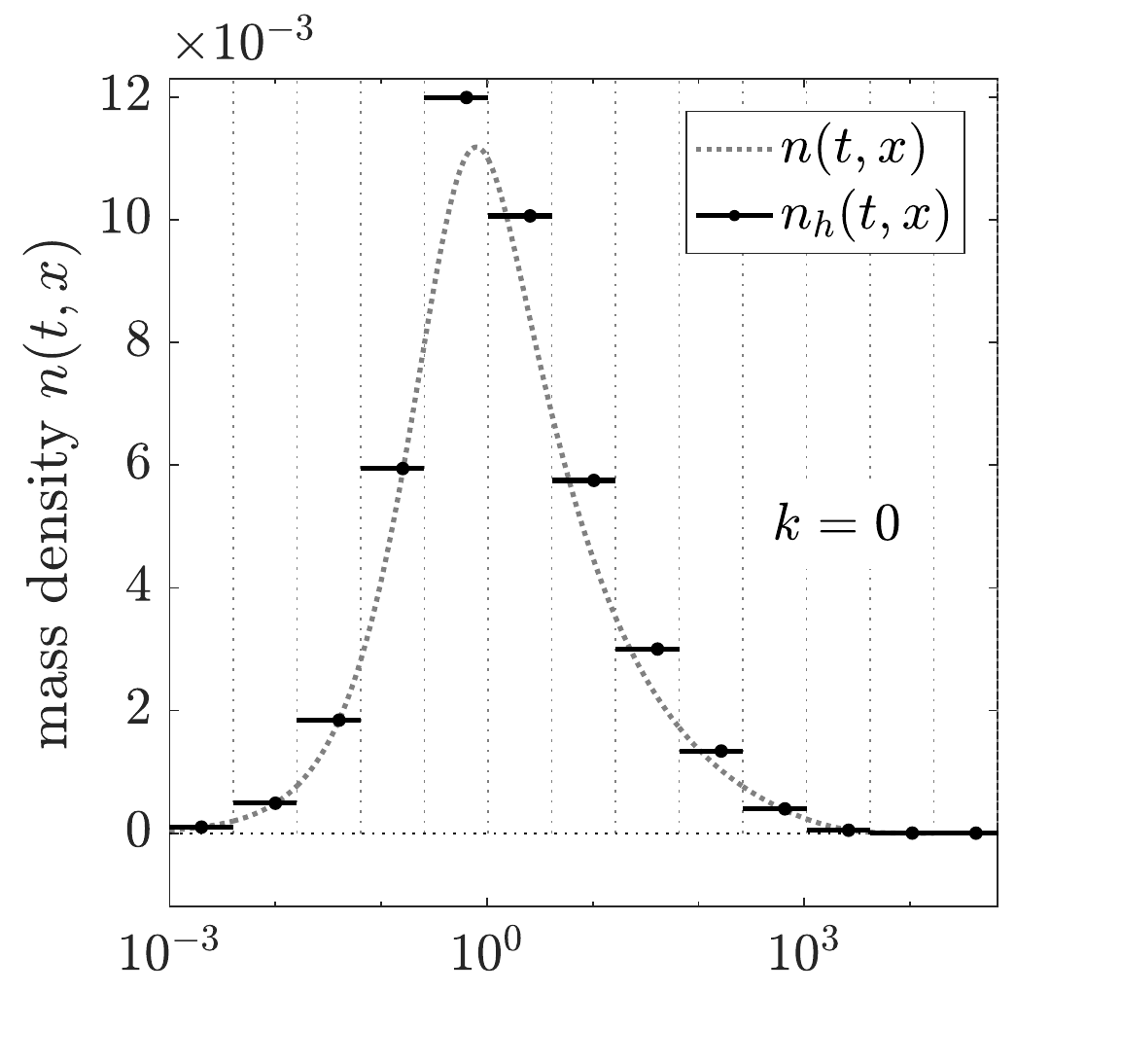} \hspace{-10mm}
\includegraphics[width=7.5cm]{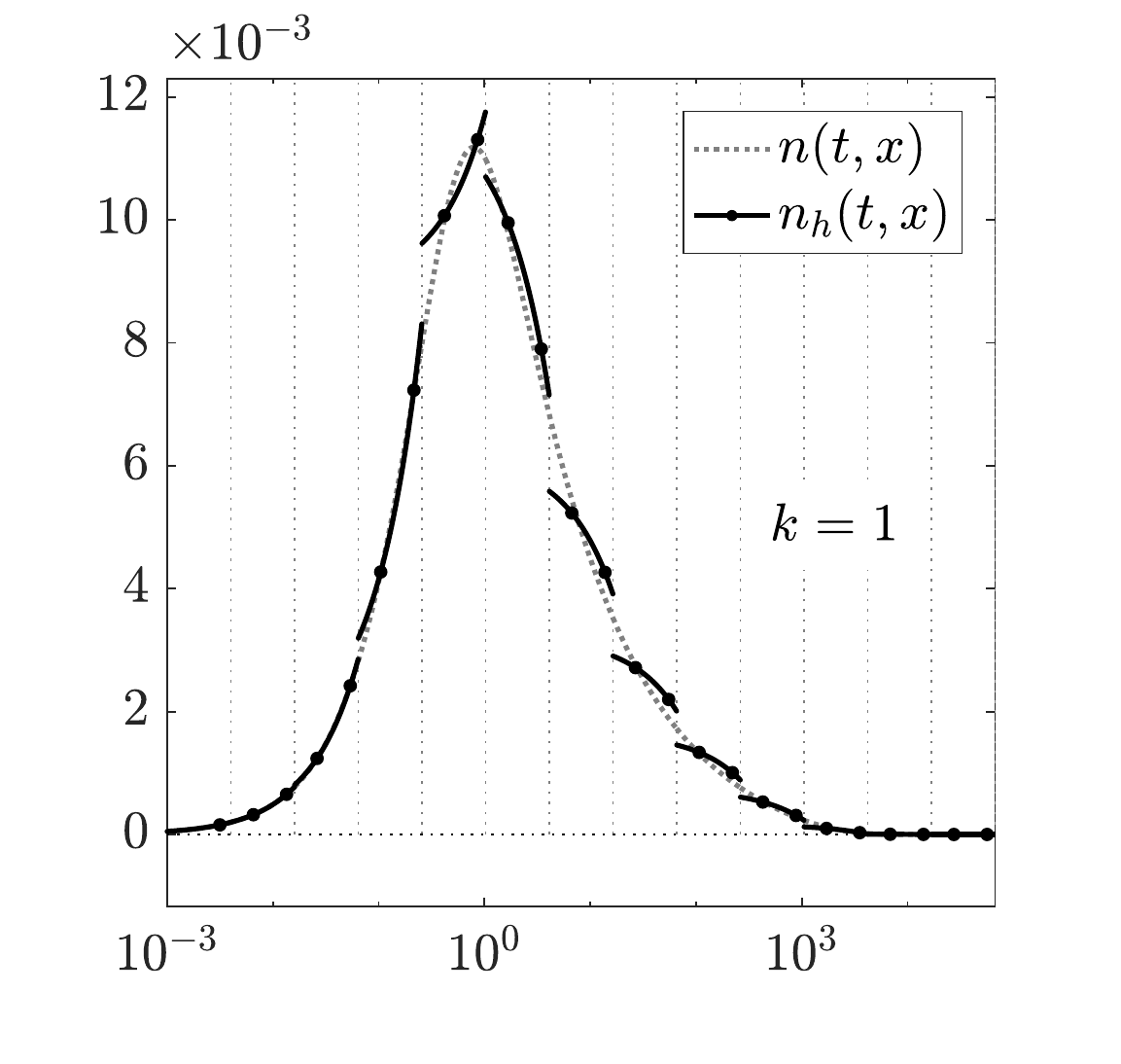} \\[-12pt]
\includegraphics[width=7.5cm]{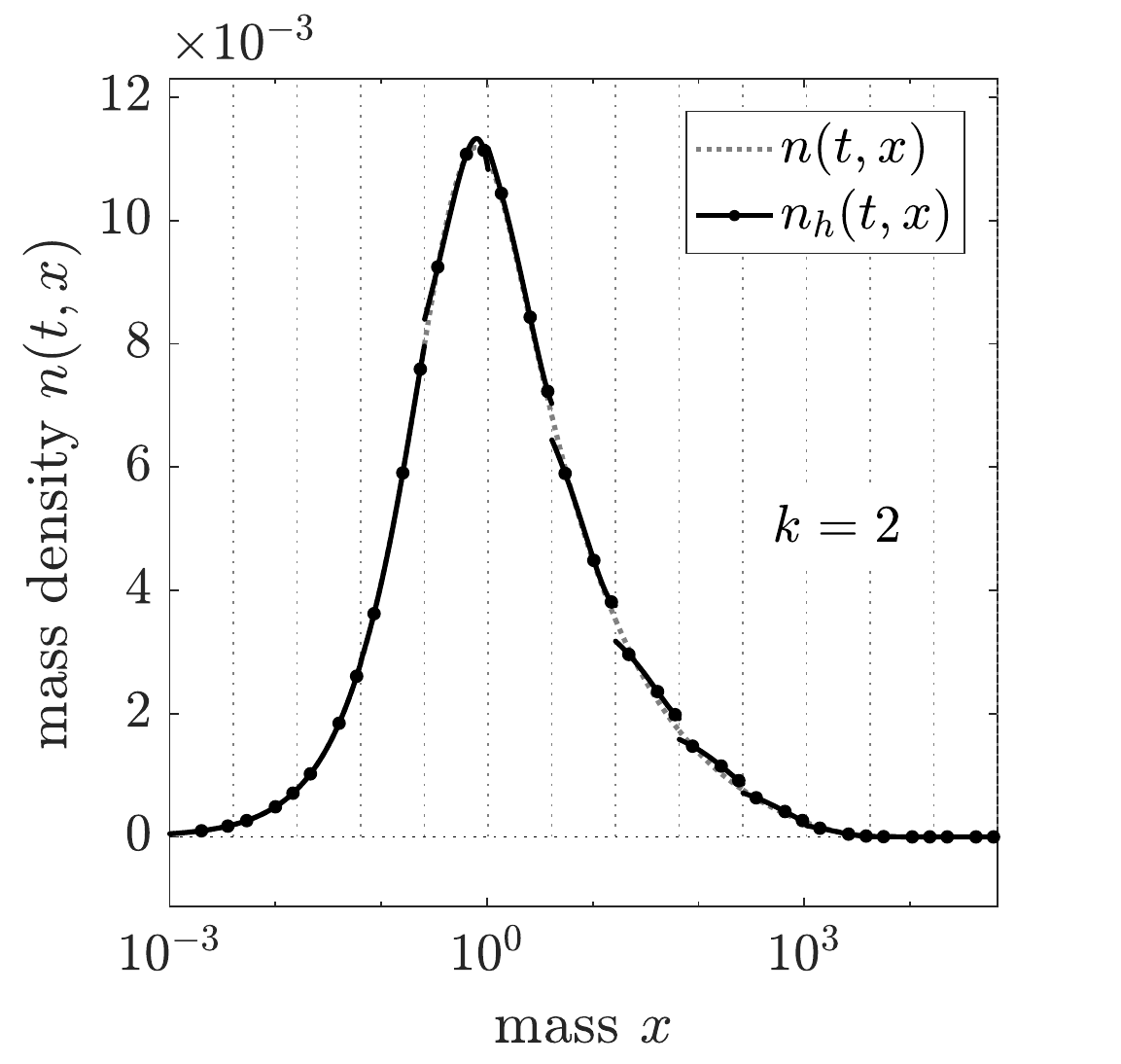} \hspace{-10mm}
\includegraphics[width=7.5cm]{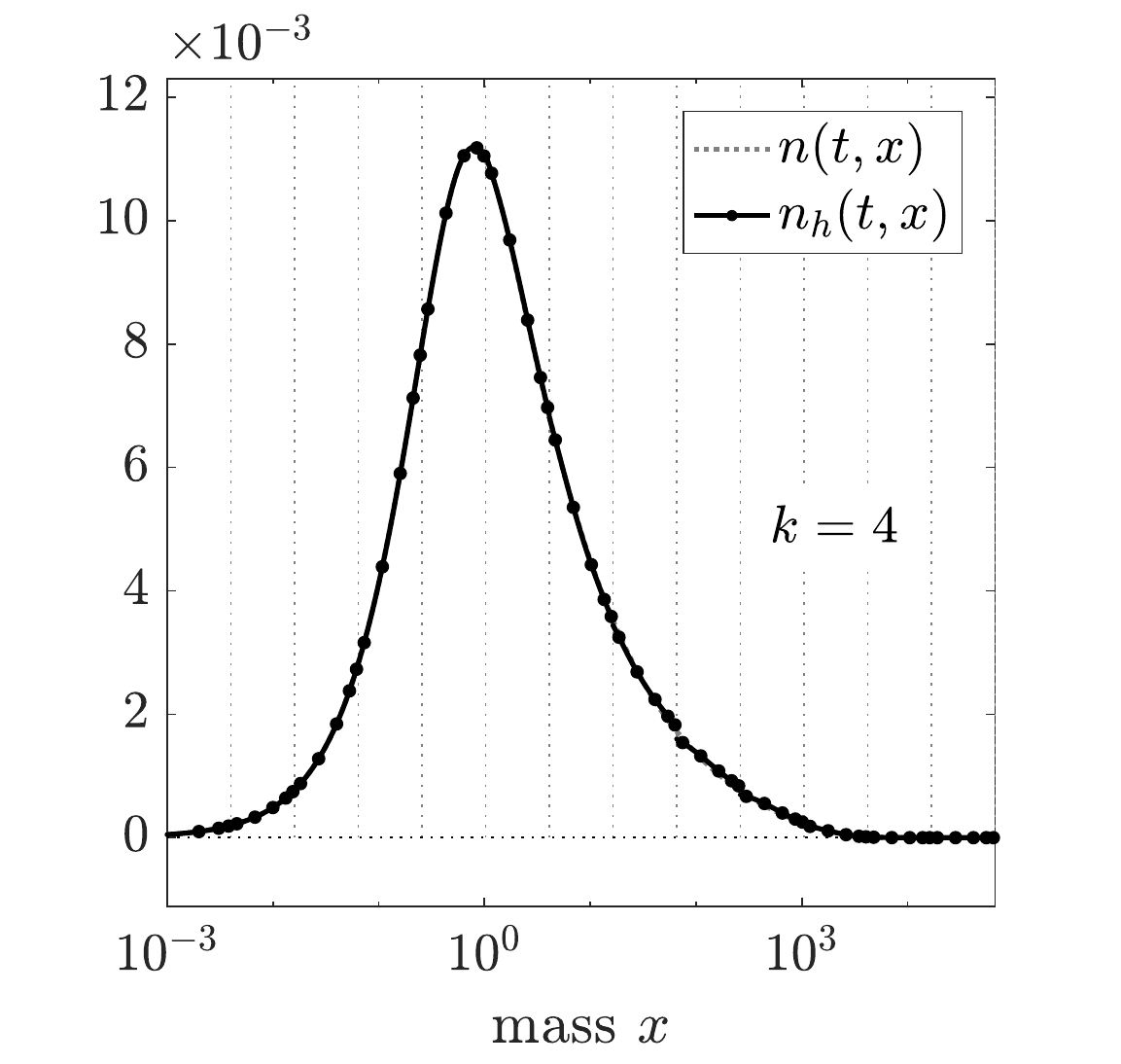}
\caption{Numerical solution for $N=15$ cells and different $k$ for test case 1}
\label{num_sol}
\end{figure}

The numerical solutions for the sum aggregation kernel with $N=15$ cells and different polynomial 
degree $k$ are shown in Figure \ref{num_sol}. The numerical solution is shown together with the 
analytical one at time $t=3$ where the degree of aggregation is 
$I_{\text{agg}}=1-M_0(t)/M_0(0)\approx 95\%$. Due to the logarithmic scale of the $x$-axis the 
piecewise linear solution for $k=1$ appears curved and also the Gauss points are of course 
symmetrically distributed in one cell. Obviously, the approximation improves for a higher 
polynomial degree.

\begin{table}[ht]
\begin{tabular}{ccccccc}
 \toprule
  $k\backslash N$ &   15   &   30    &   60    &  120    &  240    & EOC \\ \hline
        0 (FVS)   & 4.2e-1 & 2.1e-1  & 1.0e-1  & 5.2e-2  & 2.6e-2  & 1.0 \\
        1         & 1.3e-1 & 4.4e-2  & 1.1e-2  & 2.8e-3  & 6.9e-4  & 2.0 \\
        2         & 7.4e-2 & 8.0e-3  & 1.1e-3  & 1.4e-4  & 1.7e-5  & 3.0 \\
        4         & 1.3e-2 & 3.0e-4  & 1.0e-5  & 3.3e-7  & 1.0e-8  & 5.0 \\
        8         & 3.6e-5 & 3.7e-7  & 7.0e-10 & 1.4e-12 & 1.2e-14 & 9.0 \\ \hline
        CAT       & 5.3e-1 & 2.3e-1  & 1.1e-1  & 5.3e-2  & 2.6e-2  & 1.0  \\
 \bottomrule
\end{tabular} \vspace{2pt}
\caption{$L^1$ errors $e_h$ and EOC for test case 1}
\label{errors_test1}
\end{table}

\begin{table}[!ht]
\begin{tabular}{ccccccc}
 \toprule
  $k\backslash N$ &   15   &   30    &   60    &  120    &  240    & EOC \\ \hline
        0 (FVS)   & 1.3e-1 & 5.5e-2  & 1.4e-2  & 3.5e-3  & 8.8e-4  & 2.0 \\
        1         & 8.7e-2 & 9.0e-3  & 1.2e-3  & 1.5e-4  & 1.8e-5  & 3.0 \\
        2         & 3.8e-2 & 1.9e-3  & 1.1e-4  & 6.8e-6  & 4.3e-7  & 4.0 \\
        4         & 3.6e-3 & 5.2e-5  & 9.4e-7  & 1.5e-8  & 2.3e-10 & 6.0 \\
        8         & 2.9e-5 & 4.7e-8  & 6.0e-11 & 6.6e-14 & 2.0e-14 & 10.0 \\ \hline
        CAT       & 1.3e-1 & 3.7e-2  & 9.6e-3  & 2.4e-3  & 6.1e-4  & 2.0  \\
 \bottomrule
\end{tabular} \vspace{2pt}
\caption{Discrete $L^1$ errors $e_{h,d}$ and EOC for test case 1}
\label{errors_d_test1}
\end{table}

The numerical errors for the sum aggregation kernel are shown in Table \ref{errors_test1} and 
Table \ref{errors_d_test1}. In the continuous $L^1$ norm the EOC is $k+1$ as expected. 
In the discrete $L^1$ norm the EOC is $k+2$ on a geometric grid which is one order higher than 
in the continuous norm. It appears that the evaluation of the numerical solution at the same 
Gauss points as used for the discretization of the scheme shows a type of superconvergence. 
For $k=0$, the second order convergence for the Finite Volume Scheme on smooth grids was 
proven in \cite{KKW14}.

The numerical results for the constant and product aggregation kernels are very similar and 
are not shown again. In our tests we observe that even the post-gelation phase, see e.g.
\cite{EZH84}, for the product aggregation kernel can be simulated very well.

\begin{figure}[!ht]
\includegraphics[width=7.5cm]{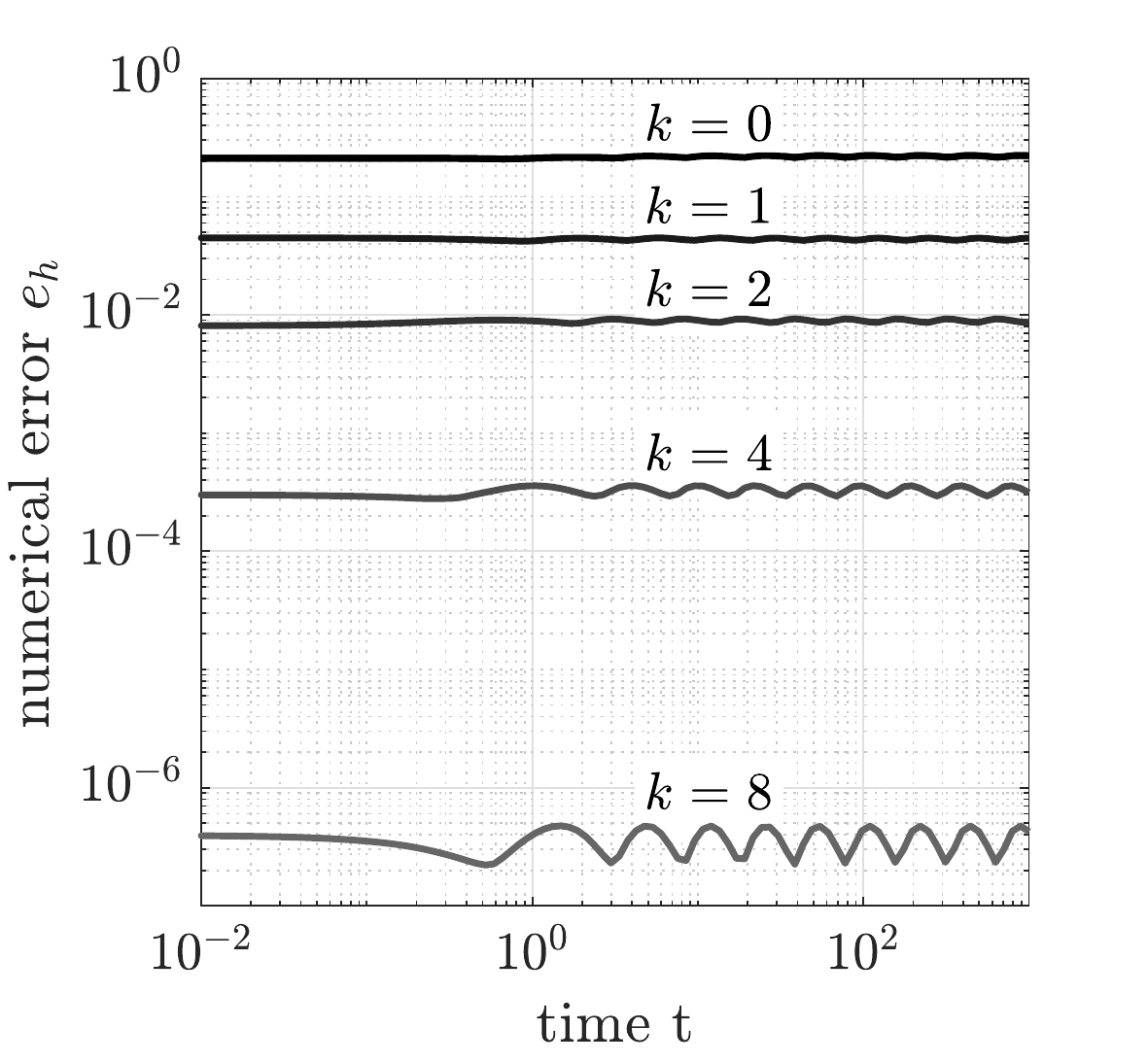} \hspace{-4mm}
\includegraphics[width=7.5cm]{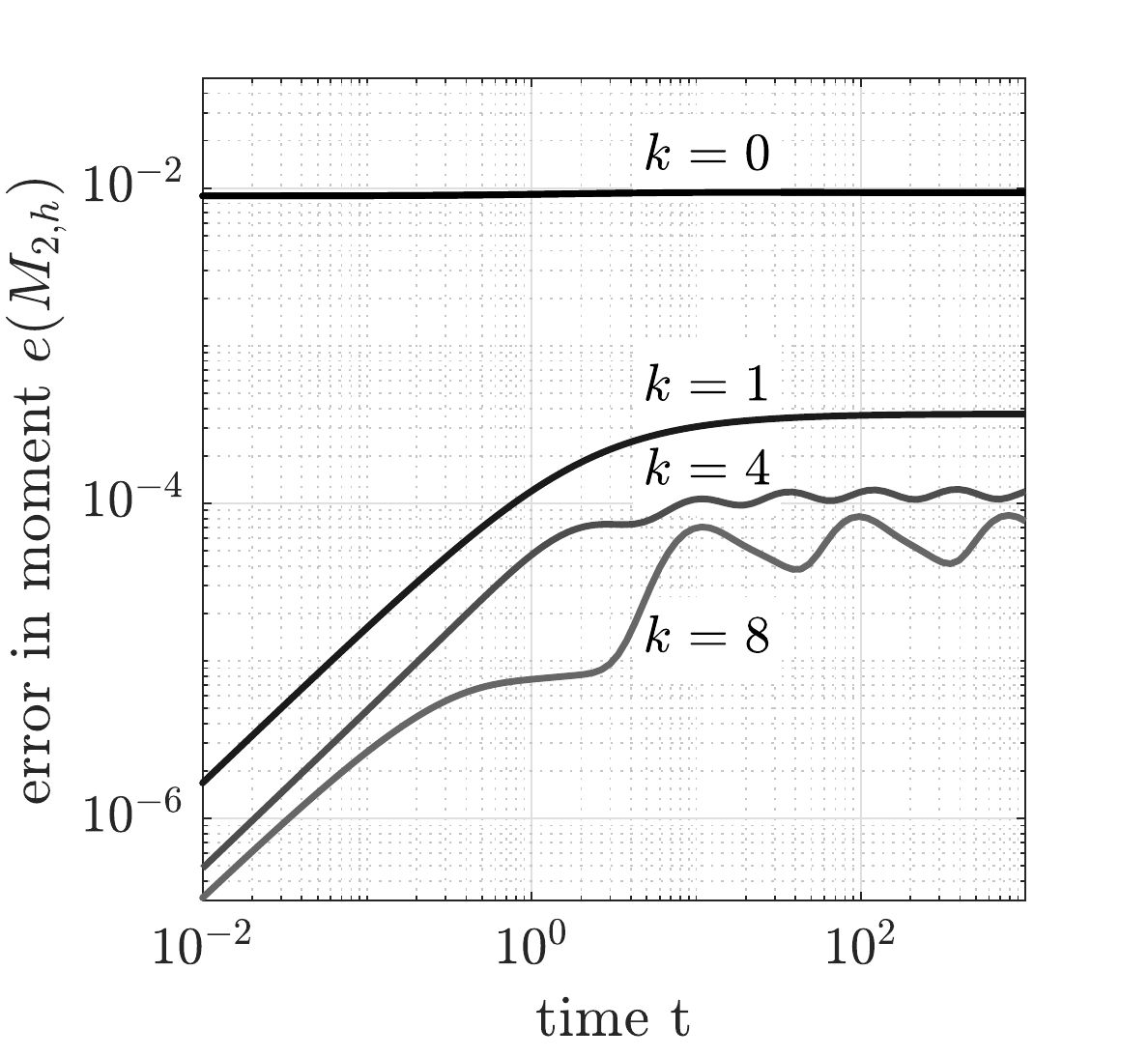}
\caption{Error evolution of the numerical solution (left) and of the second moment (right) 
for different $k$ for test case 1}
\label{error_evo}
\end{figure}

The evolution of the numerical error in time is shown for the constant aggregation kernel in 
Figure \ref{error_evo} (left). We show the results for $N=30$ cells and varying polynomial 
degree $k$ up to time $t=1000$ where the degree of aggregation is $I_{\text{agg}}\approx 99.8\%$. 
One can observe that the numerical error remains bounded for longer times.

\begin{table}[!ht]
\begin{tabular}{ccccccccc}
 \toprule
 $(N,k)$ & $e(M_{0,h})$ & $e(M_{1,h})$ & $e(M_{2,h})$ & $e(M_{3,h})$ & $e(M_{4,h})$ & $e(M_{5,h})$
 \\ \hline
 FVS (90, 0) &  6.6e-3  &    0     &  9.4e-3  &  3.7e-2  &  8.6e-2  &  1.6e-1  \\
     (45, 1) &  2.8e-4  &    0     &  3.7e-4  &  3.8e-3  &  1.6e-2  &  4.5e-2  \\
     (30, 2) &  6.7e-5  &    0     &  6.2e-4  &  2.2e-3  &  3.6e-3  &  1.5e-3  \\
     (18, 4) &  9.1e-6  &    0     &  1.2e-4  &  6.1e-4  &  3.1e-3  &  1.6e-2  \\
     (10, 8) &  5.1e-6  &    0     &  7.7e-5  &  7.6e-4  &  8.2e-3  &  4.6e-2  \\ \hline
 CAT (90)    &  3.7e-14 &    0     &  1.1e-2  &  2.9e-2  &  5.1e-2  &  7.3e-2  \\
 \bottomrule
\end{tabular} \vspace{2pt}
\caption{Numerical errors in the first 6 moments for test case 1}
\label{moments}
\end{table}

In this test case, we also discuss the approximation of the moments. The error in the first 
6 moments $M_0,\dots,M_5$ for the constant aggregation kernel at time $t=1000$ are shown in 
Table \ref{moments}. For a better comparison we have chosen the same number of evaluation points 
$N(k+1)=90$. One can see that the prediction of the zeroth moment is very accurate and even the 
higher moments are approximated well. The lower moments are predicted better for an increasing 
polynomial degree $k$. The accuracy in the higher moments for a larger polynomial degree $k$ is 
disturbed by some small oscillations of the piecewise high order polynomial for large values of 
$x$. For the Cell Average Technique (CAT) we have to use the normalized moments 
$\tilde M_{p,h}(t)=M_{p,h}(t)/M_{p,h}(0)$ for the error calculation of the moments, 
$e(\tilde M_{p,h})$, since the discretization is focused on the number density and even the 
first moment has a discretization error in the initial distribution. The first two moments are 
preserved exactly by construction of the method and it produces similar errors as the 
Finite Volume Scheme for the higher moments.

In Figure \ref{error_evo} (right) we have shown the evolution of the error in the second moment 
$e(M_{2,h})$ for different $k$ again with the same number of evaluation points $N(k+1)=90$ 
in time. One can observe that the error in the moments increases initially but remains bounded 
for longer times.

\subsection{Pure breakage} {\bf Test case 2:} \\
Here, the EOC is calculated for the binary breakage $b(x, y) = 2/y$ together with the linear and 
quadratic selection functions, i.e. $S(x) = x$ and $S(x) = x^2$. The analytical solutions for 
such problems have been given by Ziff and McGrady \cite{ZM85} for an exponential initial 
condition, $n(0, x) = x\exp(-x)$. The computational domain in these cases is taken as 
$[10^{-6}, 10^3]$.

\begin{table}[ht]
\begin{tabular}{ccccccc}
 \toprule
  $k\backslash N$ &   15   &   30    &   60    &  120    &  240    & EOC \\ \hline
        0 (FVS)   & 4.2e-1 & 2.1e-1  & 1.0e-1  & 5.2e-2  & 2.6e-2  & 1.0 \\
        1         & 1.3e-1 & 4.5e-2  & 1.1e-2  & 2.8e-3  & 6.9e-4  & 2.0 \\
        2         & 7.0e-2 & 8.0e-3  & 1.1e-3  & 1.4e-4  & 1.7e-5  & 3.0 \\
        4         & 1.3e-2 & 3.1e-4  & 1.0e-5  & 3.3e-7  & 1.0e-8  & 5.0 \\
        8         & 2.5e-5 & 4.2e-7  & 7.4e-10 & 1.4e-12 & 1.2e-14 & 9.0 \\ \hline
        CAT       & 5.3e-1 & 2.3e-1  & 1.1e-1  & 5.3e-2  & 2.6e-2  & 1.0  \\
 \bottomrule
\end{tabular} \vspace{2pt}
\caption{$L^1$ errors $e_h$ and EOC for test case 2}
\label{errors_test2}
\end{table}

The results for the linear selection function are shown in Table \ref{errors_test2}.
Hence, we observe that the DG scheme is $k+1$ order convergent in the $L^1$ norm. 
Similar to the previous case, the order of convergence is $k+2$ using the discrete $L^1$ norm 
which is not shown again. For $k=0$, the second order convergence was proven in \cite{KKW14} 
to be mesh-independent for breakage.

The results for the quadratic selection function are very similar and are omitted. Also the 
behavior of the error evolution and the error in the moments is similar and is not shown again. 
\\[14pt]
{\bf Test case 3:}
Now the case of multiple breakage with the quadratic selection function $S(x) = x^2$ is 
considered where an analytical solution is not known. For the numerical simulations, the 
following normal distribution as an initial condition is taken 
\[
n(0,x)=\frac{1}{\sigma\sqrt{2\pi}}\exp \left(-\frac{(x-\mu)^2}{2\sigma^2}\right) .
\]
The computations are made for the breakage function considered by Hill and Ng \cite{HN96}
\[
b(x,y) = p \left( \frac{[m+(m+1)(p-1)]!}{m![m+(m+1)(p-2)]!} \right) 
  \frac{x^m(y-x)^{m+(m+1)(p-2)}}{y^{pm+p-1}}, \quad p\in \mathbb{N},\ p\geq 2,
\]
where the relation $\int_0^y b(x,y)dx = p$ holds where $p$ gives the total number of fragments 
per breakage event. The parameter $m\geq 0$ is responsible for the shape of the daughter 
particle distribution. The numerical solutions are obtained using $p = 4$, $m = 2$.

\begin{table}[ht]
\begin{tabular}{ccccccc}
 \toprule
  $k\backslash N$ &   15   &   30    &   60    &  120    &  240    & EOC \\ \hline
        0 (FVS)   & 5.5e-1 & 2.5e-1  & 1.1e-1  & 6.6e-2  & 3.3e-2  & 1.0 \\
        1         & 1.7e-1 & 5.5e-2  & 1.4e-2  & 3.5e-3  & 8.8e-4  & 2.0 \\
        2         & 3.0e-2 & 1.3e-2  & 1.6e-3  & 2.0e-4  & 2.6e-5  & 3.0 \\
        4         & 1.4e-2 & 5.8e-4  & 1.6e-5  & 5.6e-7  & 1.7e-8  & 5.0 \\
        8         & 4.0e-4 & 1.5e-7  & 1.6e-9  & 3.1e-12 & 8.2e-14 & 9.0 \\ \hline
        CAT       & 3.0e-1 & 1.1e-1  & 4.6e-2  & 2.1e-2  & 1.0e-2  & 1.0  \\
 \bottomrule
\end{tabular} \vspace{2pt}
\caption{$L^1$ errors $e_h$ and EOC for test case 3}
\label{errors_test3}
\end{table}

For the numerical simulation the computational domain is taken as $[10^{-6}, 10^3]$.
In this case, the numerical error $e_h$ is computed from the numerical solutions $n_h$ and 
$n_{h/2}$, as mentioned above. As expected, we again observe from Table \ref{errors_test3} 
that the DG scheme shows convergence of order $k+1$. Using the discrete $L^1$ norm we again 
obtain the higher convergence order $k+2$.

\subsection{Coupled aggregation-breakage} {\bf Test case 4:} \\
Finally, the EOC is evaluated for the simultaneous process with a constant aggregation kernel 
$K(x,y) = 1$ and breakage kinetics $b(x,y) = 2/y$, $S(x) = x/2$.
For the simulation the computational domain $[10^{-3}, 10^6]$ is taken. 
The analytical solutions for this problem are given by Lage \cite{La02} for the following two 
different initial conditions
\begin{align*}
& n(0,x) = x\mathrm{e}^{-x}, \\
& n(0,x) = 4x^2\mathrm{e}^{-2x}.
\end{align*}
The former exponential initial condition is a steady state solution. The latter Gaussian-like 
initial condition is a special case where the number of particles stays constant. From Table 
\ref{errors_test4}, we find that the DG scheme is $k+1$ order convergent where the results for 
the first case are shown. As before, we obtain one order higher convergence on a geometric grid 
in the discrete norm. The second case is very similar and is omitted.

\begin{table}[ht]
\begin{tabular}{ccccccc}
 \toprule
  $k\backslash N$ &   15   &   30    &   60    &  120    &  240    & EOC \\ \hline
        0 (FVS)   & 4.2e-1 & 2.1e-1  & 1.1e-1  & 5.2e-2  & 2.6e-2  & 1.0 \\
        1         & 1.3e-1 & 4.5e-2  & 1.1e-2  & 2.8e-3  & 6.9e-4  & 2.0 \\
        2         & 7.4e-2 & 8.1e-3  & 1.1e-3  & 1.4e-4  & 1.7e-5  & 3.0 \\
        4         & 1.3e-2 & 3.0e-4  & 1.0e-5  & 3.3e-7  & 1.0e-8  & 5.0 \\
        8         & 3.9e-5 & 4.0e-7  & 7.3e-10 & 1.4e-12 & 4.2e-15 & 9.0 \\ \hline
        CAT       & 5.3e-1 & 2.3e-1  & 1.1e-1  & 5.3e-2  & 2.6e-2  & 1.0  \\
 \bottomrule
\end{tabular} \vspace{2pt}
\caption{$L^1$ errors $e_h$ and EOC for test case 4}
\label{errors_test4}
\end{table}

\section{Concluding remarks}\label{CONC}
In this paper, we have developed a high order DG scheme which can be proven to be 
positivity-preserving for both coagulation and fragmentation equations.
We have tested the DG scheme and clearly observed the convergence order $k+1$ and the strict 
positivity preservation in all these tests. 
Interestingly, the DG scheme shows a type of superconvergence of order $k+2$ on a geometric 
grid in the discrete norm which evaluates the numerical solution at the same Gaussian 
quadrature points as used for the discretization of the scheme.
Even though the CFL condition derived to preserve positivity might be very small in some
cases, we emphasize that it is not a necessary condition and can be relaxed significantly.
The high accuracy was verified numerically by taking various examples of pure aggregation, 
pure breakage and the combined problems.

By applying the limiter or the simplified version which avoids the evaluation of extrema of 
polynomials, to a discontinuous Galerkin scheme solving one dimensional 
coagulation-fragmen\-tation equations, with the time evolution by a SSP Runge-Kutta method,
we obtain a high order accurate scheme with strict positivity-preservation.

\subsection*{Acknowledgements} 
Liu's research was partially supported by the National Science Foundation under grant 
DMS-1312636, and DMS-1107291. Gr\"opler's research was partially supported by the 
German Research Foundation DFG, Research Training Group GRK 1554.

\section*{Appendix}\label{Appendix}
In Section \ref{subsec:flux} we have shown the flux evaluation for the boundary term in the 
DG scheme \eqref{dg}. For the numerical approximation of the second term we apply Gaussian 
quadrature, as already shown in (\ref{intF}). The approximation of the flux at the Gauss points, 
$F(t,\hat x_{j}^\gamma)$, is shown in the following. For aggregation the flux 
$F_a(t,\hat x_{j}^\gamma)$ can be approximated by the numerical flux
\begin{equation}
F_{a,j}^\gamma = \sum_{l=1}^{j} \frac{b_l-a_l}{2}\sum_{\alpha=1}^Q \omega_\alpha 
n_h(t,u_{l}^\alpha) \Gamma_{j,l}^{\gamma,\alpha},
\end{equation}
where $a_l=x_{l-1/2}$ and $b_l=x_{l+1/2}$ for $l<j$ and $b_l=\hat x_{j}^\gamma$ for $l=j$, 
and $u_l^\alpha=\frac{1}{2}(b_l+a_l)+ \frac{1}{2}(b_l-a_l) s_\alpha$.
The partial flux $\Gamma_j^\gamma(u_l^\alpha)$ is approximated by
\begin{equation}
\Gamma_{j,l}^{\gamma,\alpha} = 
\frac{1}{2}(b_J-a_J) \sum_{\beta=1}^Q \omega_\beta A(u_l^\alpha,y_J^\beta) n_h(t,y_J^\beta) 
+ \sum_{i=J+1}^N \frac{h_i}{2} \sum_{\beta=1}^Q \omega_\beta A(u_l^\alpha,\hat x_i^\beta)
n_h(t,\hat x_i^\beta),
\end{equation}
where $a_J=\hat x_{j}^\alpha- u_{l}^\alpha$, $b_J=x_{J+1/2}$, and
$y_J^\beta=\frac{1}{2}(b_J+a_J)+ \frac{1}{2}(b_J-a_J) s_\beta$.
The index $J$ is chosen such that $\hat x_j^\gamma -u_l^\alpha\in I_J$.
Here, the index $J$ depends on $j$, $\gamma$, $l$ and $\alpha$.

For breakage the flux $F_b(t,\hat x_{j}^\gamma)$ is approximated by
\begin{equation}
F_{b,j}^\gamma = -\sum_{l=j}^{N} \frac{b_l-a_l}{2} \sum_{\alpha=1}^Q \omega_\alpha 
n_h(t,u_{l}^\alpha) G_{j,l}^{\gamma,\alpha},
\end{equation}
where $a_l=\hat x_{j}^\gamma$ for $l=j$ and $a_l=x_{l-1/2}$ for $l>j$, $b_l=x_{l+1/2}$, and
$u_l^\alpha=\frac{1}{2}(b_l+a_l)+ \frac{1}{2}(b_l-a_l) s_\alpha$.
The partial flux $G_j^\gamma(u_l^\alpha)$ is approximated by
\begin{equation}
G_{j,l}^{\gamma,\alpha}
= \sum_{i=1}^j \frac{b_i-a_i}{2} \sum_{\beta=1}^Q \omega_\beta B(u_{i}^\beta,u_{l}^\alpha),
\end{equation}
where $a_i=x_{i-1/2}$, $b_i=x_{i+1/2}$ for $i<j$ and $b_i=\hat x_{j}^\gamma$ for $i=j$, and
$u_i^\beta=\frac{1}{2}(b_i+a_i)+ \frac{1}{2}(b_i-a_i) s_\beta$.

\end{document}